\documentclass{article}
\usepackage[T1]{fontenc}                % Paquete de español para los acentos y demás caracteres.
\usepackage{amsmath,latexsym,amssymb}   % Para los símbolos matemáticos.
\usepackage{theorem}                    % Para los teoremas, proposiciones...

\theoremstyle{break}                                   % Después de teorema pasa a la siguiente linea para el enunciado.
\newtheorem{Teo}{\textbf{Theorem}}[section]
\newtheorem{Def}[Teo]{\emph{\textbf{Definition}}}
\newtheorem{Ej}[Teo]{\textbf{Example}}
\newtheorem{Pro}[Teo]{\textbf{Proposition}}
\newtheorem{Lem}[Teo]{\textbf{Lemma}}
\newtheorem{Cor}[Teo]{\textbf{Corollary}}

\newcommand{\Beachy}{Beachy-Blair }
\newcommand{\Zel}{Zelmanowitz}
\newcommand{\K}{\mathbb{K}}
\newcommand{\Z}{\mathbb{Z}}

\newcommand{\binv}{\setminus}
\newcommand{\0}{\{0\}}

\newcommand{\tab}{\hspace*{0.5cm}}
\newcommand{\bc}{\begin{center}}
\newcommand{\ec}{\end{center}}
\newcommand{\bit}{\begin{itemize}}
\newcommand{\eit}{\end{itemize}}
\newcommand{\ben}{\begin{enumerate}}
\newcommand{\een}{\end{enumerate}}

\newcommand{\DCC}{DCC$\bot$}

\newcommand{\ds}{\displaystyle}

\newenvironment{proof}{\par\noindent{\bf Proof. }}{$\cqd$\par\bigskip}
\newcommand{\cqd}{\enspace\vrule height6pt width4pt depth2pt}
\begin{document}
\title{Rings with the \Beachy condition \thanks{ Research partially supported by grants of MICIN-FEDER (Spain) MTM2008-06201-C02-01, Generalitat de
Catalunya 2009SGR1389.}}

\author{Elena Rodríguez-Jorge}
\date{}
\maketitle

\begin{abstract}
A ring satisfies the left \Beachy condition if each of its faithful left ideal is cofaithful. Every left zip ring satisfies the left
\Beachy condition, but both properties are not equivalent. In this paper we will study the similarities and the differences between zip rings and rings
with the \Beachy condition. We will also study the relationship between the \Beachy condition of a ring and its skew polynomial and skew power series
extensions. We give an example of a right zip ring that is not left zip, proving that the zip property is not symmetric.
\end{abstract}



\section{Introduction}

\tab The first time that the concept of zip ring appeared as it is known nowadays was in 1989,
by Faith in \cite{Faith}. Previously, Beachy and Blair in \cite{Beachy} (1975) and \Zel\ in \cite{Zelm} (1976), introduced a more general property.
In \cite{Beachy}, Beachy and Blair defined rings whose faithful left ideals are cofaithful (we will call these rings, rings with the left \Beachy
condition) and, in \cite{Zelm}, \Zel\ worked with rings with the ``finite intersection property'' on annihilator left ideals. Both properties are
equivalent, but they were introduced independently and parallelly, obtaining quite different results.

\Zel, in \cite{Zelm}, noted that his condition was less restrictive than \DCC\ (descending chain condition on annihilators), that is, there exist rings
with the left \Beachy condition that do not satisfy the left \DCC, but every ring satisfying the left \DCC\ satisfies the left \Beachy condition.
In fact, the reason why \Zel\ introduced his property was to weaken the chain condition on annihilators. From the point of view of Beachy and Blair,
the \Beachy condition arose in order to give a characterization of semiprime left Goldie rings. They proved in \cite{Beachy} that a semiprime ring is
left Goldie (that is, it satisfies the ascending chain condition on left annihilators and it has finite uniform dimension) if and only if it satisfies
the left \Beachy condition and that every nonzero left ideal contains a nonzero uniform left ideal.

Let us recall some basic definitions.
\begin{Def}
\tab A ring $R$ is left zip if for every subset $X\subseteq R$ such that $l.ann_R(X) = \{0\}$, there exists a finite subset $F\subseteq X$
such that $l.ann_R(F) = \{0\}$, where $l.ann_R(X) = \{r\in R\mid rx = 0\text{ for all } x\in X\}$ denotes the left annihilator of $X$ in $R$.
\end{Def}
\tab Analogously we can define right zip ring. A ring is zip if it is both left and right zip.

\begin{Def}
\tab A ring $R$ satisfies the left \Beachy condition if for every faithful left ideal $I$ of $R$ (that is, $l.ann_R(I) = \{0\}$), there exists a
finite subset $F\subseteq I$ such that $l.ann_R(F) = \{0\}$.
\end{Def}
\tab Analogously we can define the right \Beachy condition. A ring satisfies the \Beachy condition if it satisfies both the left and the right \Beachy
conditions.

Throughout this paper, all rings are supposed to be associative with identity. Unless otherwise stated, all results are given on the left but are also
true on the right.

%%
%\begin{Remark}
%Let $(R,\cdot,+)$ be an associative ring. We define the ring $(R', *, +)$ as the set $R' = R$, with the same additive operation than $(R,\cdot,+)$ and
%with the product given by $a * b = b\cdot a$ for all $a,b\in R'$.
%
%Note that $l.ann_R(r) = r.ann_{R'}(r)$ and $r.ann_R(r) = l.ann_{R'}(r)$ for all $r\in R$. In addition, $(R[x])' = R'[x]$.
%\end{Remark}
%%

Every left zip ring satisfies the left \Beachy condition, but there are examples of rings with the left \Beachy condition that are not left zip
($R[[x]]$ in Proposition \ref{prob2} is an example of this kind of rings). However, for commutative or reduced rings, it is not difficult to see that
both properties are equivalent. In general, we have the following:
\begin{center}
Left \DCC  $\begin{array}{c}
                        \\
                        \Rightarrow \\
                        \nLeftarrow
             \end{array}$ Left Zip $\begin{array}{c}
                                                     \\
                                                     \Rightarrow \\
                                                     \nLeftarrow
                                      \end{array}$ Left \Beachy condition.
\end{center}

Faith, in \cite{Faith2}, proposed the following questions regarding zip rings.\\

Let $R$ be any ring.
\ben
\item Does $R$ being a left zip ring imply $R[x]$ being left zip?
\item Does $R$ being a left zip ring imply $M_n(R)$ being left zip?
\item Does $R$ being a left zip ring imply $R[G]$ being left zip when $G$ is a finite group?
\een

Cedó in \cite{Cedo} (1991) answered all these questions in the negative. However, when $R$ is a commutative ring, Beachy and Blair
(\cite[Proposition 1.9]{Beachy}) gave a positive answer to 1. and Cedó (\cite[Proposition 1]{Cedo}) gave a positive answer to 2. In \cite{Faith2},
Faith proved that if $R$ is a commutative zip ring and $G$ is a finite abelian group, then the group ring $R[G]$ is zip.

It is natural then to ask these same questions for rings with the left \Beachy condition.

Beachy and Blair proved that the \Beachy condition is Morita invariant (\cite[Corollary 1.2]{Beachy}), and, therefore, a ring
$R$ satisfies the left \Beachy condition if and only if $M_n(R)$ satisfies the left \Beachy condition for all $n\geq 1$, so the answer to 2. is positive
for rings with the left \Beachy condition, even in the noncommutative case.

Note that the example of Cedó of a domain $S$ such that $M_n(S)$ is not right zip (\cite[Example 1]{Cedo}), gives us an example of a ring, $M_n(S)$,
that satisfies the right \Beachy condition, since $S$ satisfies the right \Beachy condition, but is not right zip.

It is widely believed that the answer to 1. for rings with the left \Beachy condition should be negative in general, but it remains as an open problem,
since no counterexample has been found so far. However, in Section \ref{poli} we will prove that, under certain conditions, the answer to 1. is positive
for rings with the left \Beachy condition. In Section \ref{series} we will study the relationship between the \Beachy condition of a ring and its skew
power series extension, and compare our results with similar known results for zip rings. Finally, in Section \ref{ejemplos}, we will construct an
example that answers in the negative some open problems regarding the \Beachy condition and the zip property.

\section{Skew polynomial extensions over rings with the \Beachy condition}\label{poli}

\tab In this section we will study the relationship between the \Beachy condition of a ring and its skew polynomial extension.
Let $R$ be a ring and $\alpha$ be an endomorphism of $R$. The $\alpha$-skew polynomial extension of $R$, denoted by $R[x;\alpha]$, is the ring with
elements of the form $\ds{\sum_{i=0}^n} a_ix^i$, with $a_i\in R$, and with the multiplication defined by
$$(\sum_{i=0}^n a_ix^i)(\sum_{j=0}^m b_jx^j) = \sum_{i=0}^n\sum_{j=0}^m a_i\alpha^i(b_j) x^{i+j}$$
and the sum defined by
$$(\sum_{i=0}^n a_ix^i)+(\sum_{j=0}^m b_jx^j) = \sum_{i=0}^{\max\{n,m\}} (a_i+b_i) x^i.$$
\tab In particular, $xb = \alpha(b)x$ for all $b\in R$.

\begin{Def}
\tab Let $R$ be a ring and $\alpha$ be an endomorphism of $R$. $R$ is $\alpha$-skew Armendariz if for all $\ds{f(x)=\sum_{i=0}^n a_ix^i}$,
$\ds{g(x)=\sum_{j=0}^n b_jx^j}$ $\in R[x;\alpha]$ such that $f(x)g(x)=0$ we have that $a_i\alpha^i(b_j)=0$ for all $i,j$. When $\alpha$ is the
identity of $R$, we say that $R$ is an Armendariz ring.
\end{Def}

Let $\alpha$ be an automorphism of $R$. Let $\Gamma = \{l.ann_R(U)\mid U\subseteq R\}$ and $\Delta = \{l.ann_{R[x;\alpha]}(V) \mid V\subseteq
R[x;\alpha]\}$. Since for all $U\subseteq R$ we have that $l.ann_{R[x;\alpha]}(U)=R[x;\alpha]l.ann_R(U)$, we can define the map
$\phi: \Gamma\to\Delta$ by $\phi(A)=R[x;\alpha]A$, for all $A\in \Gamma$.

On the other hand, for all $V\subseteq R[x;\alpha]$ we define $C_V\subseteq R$ by $C_V = \ds{\bigcup_{f(x)\in V}}C_f$ and
$C_f = \{a_0, a_1, \dots, a_n\}\cup\{0\}$, where $f(x)=\ds{\sum_{i=0}^n a_ix^i}$. Then, since $l.ann_{R[x;\alpha]}(V)\cap R = l.ann_R(C_V)$,
we can define the map $\psi:\Delta\to\Gamma$ by $\psi(B)=B\cap R$ for all $B\in \Delta$.

It is easy to see that $\phi$ is always injective and that $\psi$ is always surjective.
Moreover, Cortes, in \cite{Cortes2}, noted that $\phi$ is bijective if and only if $\psi$ is bijective, and, in this case, one is the inverse of the
other, and proved, similarly as Hirano did in \cite{Hirano} for polynomial rings, that this happens if and only if the ring $R$ is $\alpha$-skew
Armendariz (\cite[Lemma 2.7]{Cortes2}).

Note that neither the zip property nor the \Beachy condition pass to subrings in general. However, Cortes in \cite{Cortes2} proved that
if $R[x;\alpha]$, where $\alpha$ is an automorphism of a ring $R$, is left zip, then the ring $R$ is left zip as well. In the following Lemma we will
see that, although we can't prove the same for the \Beachy condition, adding another assumption, a similar result can be proven.
\begin{Lem}\label{lema1}
\tab Let $R$ be a ring and $\alpha$ be an automorphism of $R$. Then, if $R[x;\alpha]$ satisfies the left \Beachy condition and $R$ is $\alpha$-compatible
(i.e. for all $a,b\in R$, $ab = 0$ if and only if $a\alpha(b) = 0$), then $R$ satisfies the left \Beachy condition.
\end{Lem}

\begin{proof}
Let $I$ be a left ideal of $R$ such that $l.ann_R(I)=\{0\}$. Then $l.ann_{R[x;\alpha]}(I)=\phi(l.ann_R(I))=\phi(\{0\}) = \{0\}$.
Moreover,
$$l.ann_{R[x;\alpha]}(R[x;\alpha]I)\subseteq l.ann_{R[x;\alpha]}(I)=\{0\}$$
and $R[x;\alpha]I$ is a left ideal of $R[x;\alpha]$. Then, since $R[x;\alpha]$ satisfies the left \Beachy condition, there exists a finite subset
$Y\subseteq R[x;\alpha]I$ such that $l.ann_{R[x;\alpha]}(Y)=\{0\}$.
Let $X = C_Y$ (the subset containing all the coefficients of all polynomials in $Y$). Let $a\in l.ann_R(X)$. Then, for all $f(x)\in Y$, we have that
$af(x) = 0$, so $a\in l.ann_{R[x;\alpha]}(Y)=0$. Therefore, $l.ann_R(X) = \{0\}$.

Note that, since $I$ is a left ideal of $R$, every $g(x)\in R[x;\alpha]I$ is of the form $g(x) = \ds{\sum_{i=0}^n} \alpha^i(a_i) x^i$, where $a_i\in I$.
Assume $X = \{\alpha^{i_1}(a_1),\dots, \alpha^{i_m}(a_m)\}$ for some $X' = \{a_1,\dots, a_m\}\subseteq I$. If $a\in l.ann_R(X')$ then, since $R$ is
$\alpha$-compatible, $a\alpha^{i_k}(a_k) = 0$ for all $1\leq k\leq m$, so $a\in l.ann_R(X) = \{0\}$. Therefore, $l.ann_R(X') = \{0\}$ and $R$ satisfies
the left \Beachy condition.
\end{proof}

Cortes in \cite[Theorem 2.8]{Cortes2} proved that, when the ring $R$ is $\alpha$-skew Armendariz, $R$ is left zip if and only if $R[x;\alpha]$ is left
zip. A similar result is also true for rings with the \Beachy condition.

\begin{Teo}\label{Beachy skew}
\tab Let $R$ be an $\alpha$-skew Armendariz ring with $\alpha$ an automorphism of $R$.

If $R$ satisfies the left \Beachy condition then
$R[x;\alpha]$ satisfies the left \Beachy condition. Moreover, if $R$ is $\alpha$-compatible, then the converse is true.
\end{Teo}

\begin{proof}
We will denote by $S$ the skew polynomial ring $S = R[x;\alpha]$.
Suppose that $R$ satisfies the left \Beachy condition. Let $J\subseteq S$ be a left ideal of $S$ such that $l.ann_S(J)=\{0\}$. Then, if $J' = JS$ is the
(two-sided) ideal of $S$ generated by $J$, we have that $l.ann_S(J') = \{0\}$.
Then, $l.ann_R(C_{J'})=\psi (l.ann_S(J')) = \psi(\{0\}) = \{0\}$ and $C_{J'}\subseteq R$. We shall see that $C_{J'}$ is an ideal of $R$.
\renewcommand{\labelenumi}{(\arabic{enumi})}
\ben
\item If $s\in C_{J'}$, then there exists $f(x)\in J'$ such that $\ds{f(x) = \sum_{i=0}^n a_ix^i}$ and $s = a_i$ for some $i$.
Let $r\in R$. Since $J'$ is an ideal of $S$, $g(x) = f(x)\alpha^{-i}(r), h(x) = rf(x)\in J'$, and the coefficients of $x^i$ in $g(x)$ and $h(x)$ are
$a_ir$ and $ra_i$ respectively. Therefore $sr, rs\in C_{J'}$.

\item If $r_1,r_2\in C_{J'}$, there exist $f(x),g(x)\in J'$ such that $\ds{f(x) = \sum_{i=0}^n a_ix^i}$, $\ds{g(x) = \sum_{j=0}^m b_jx^j}$ and
$r_1 = a_i$ for some $i$, $r_2 = b_j$ for some $j$. We want to see that $r_1+r_2\in C_{J'}$. Assume without loss of generality that $i\leq j$.
Since $f(x),g(x)\in J'$ and $J'$ is an ideal of $S$ we have that $h(x) = f(x)x^{j-i}+g(x)\in J'$, and the coefficient of $x^j$ in $h(x)$ is $a_i+b_j$,
so $r_1+r_2\in C_{J'}$.
\een

Now, since $R$ satisfies the left \Beachy condition, there exists a finite subset $X\subseteq C_{J'}$ such that $l.ann_R(X)=\{0\}$.
Assume that $X = \{a_1,\dots,a_n\}$, then, for every $a_i$ there exists a polynomial $f_i(x)$ in $J'$ such that $a_i\in C_{f_i}$. Let
$Y = \{f_1(x),\dots,f_n(x)\}\subseteq J'$. Clearly $X\subseteq C_Y$, so $l.ann_R(C_Y)\subseteq l.ann_R(X)=\{0\}$. Now, $l.ann_R(C_Y) = \psi(l.ann_S(Y)) =
\{0\}$, and, since $R$ is $\alpha$-skew Armendariz, by \cite[Lemma 2.7]{Cortes2}, $\psi$ is bijective, so $l.ann_S(Y)=\{0\}$.

By the definition of $J'$, there exist integers $m_1,\dots m_n\geq 0$, and polynomials $f_{i,j}(x)\in J$, $s_{i,j}(x)\in S$ for all $1\leq i\leq n$,
$0\leq j\leq m_i$ such that
$$f_i(x) = \sum_{j=0}^{m_i} f_{i,j}(x) s_{i,j}(x).$$
\tab Let $Y' = \{f_{i,j}(x)\mid 1\leq i\leq n\text{ and } 0\leq j\leq m_i\}\subseteq J$. Clearly, $l.ann_S(Y')\subseteq l.ann_S(Y) = \{0\}$.
Therefore, $S$ satisfies the left \Beachy condition.

If $R$ is $\alpha$-compatible, then the converse follows by Lemma~\ref{lema1}.
\end{proof}

An easy consequence of Theorem~\ref{Beachy skew} is the following result.
\begin{Cor}\label{Corbeachy}
\tab Let $R$ be an Armendariz ring. Then, $R$ satisfies the left \Beachy condition if and only if $R[x]$ satisfies the left \Beachy condition.
\end{Cor}

Although in 1991 Cedó proved that there exist right zip rings $R$ such that $R[x]$ is not right zip, it is still an open problem whether there exists
such an example for the \Beachy condition or not.

\section{Skew power series extensions over rings with the \Beachy condition}\label{series}

\tab In this section we will study the relationship between the \Beachy condition of a ring and its skew power series extension.
Let $R$ be a ring and $\alpha$ be an endomorphism of $R$. The $\alpha$-skew power series extension of $R$, denoted by $R[[x;\alpha]]$, is the ring with
elements of the form $\ds{\sum_{i\geq 0}} a_ix^i$, with $a_i\in R$, and with the multiplication defined by
$$(\sum_{i\geq 0} a_ix^i)(\sum_{j\geq 0} b_jx^j) = \sum_{i\geq 0}\sum_{j\geq 0} a_i\alpha^i(b_j) x^{i+j}$$
and the sum defined by
$$(\sum_{i\geq 0} a_ix^i)+(\sum_{j\geq 0} b_jx^j) = \sum_{i\geq 0} (a_i+b_i) x^i.$$
\tab In particular, $xb = \alpha(b)x$ for all $b\in R$.

First of all, it is important to remind that it remains as an open problem whether or not the \Beachy condition passes to the power series ring in
general. For zip rings, Cedó in \cite[Example 2]{Cedo}, proved that, for any field $\K$, there exists a right zip $\K$-algebra $R$ such that $R[x]$ is
not right zip. We will prove that this example of Cedó also satisfies that $R[[x]]$ is not right zip.

\begin{Ej}\label{ejemplocedo}
\tab For any field $\K$, there exists a right zip $\K$-algebra $R$ such that $R[[x]]$ is not right zip.
\end{Ej}

\begin{proof}
Recall the construction of the example of Cedó \cite[Example 2]{Cedo}.

Let $\K$ be a field. Let $R$ be the $\K$-algebra with set of generators $A = \{a_\infty, a_\lambda, a_{0,n}, a_{1,n}, b_{1,n}, b_{2,n}\mid n\geq 0,
\lambda\in\K\}$ and with relations:
\ben
\renewcommand{\labelenumi}{(\roman{enumi})}
\renewcommand{\theenumi}{(\roman{enumi})}
\item $a_{0,i}b_{1,j} = a_{0,i}b_{2,j} = a_{1,i}b_{1,j}$ for all $j\geq i\geq 0$,
\item $a_{1,i}b_{2,j} = 0$ for all $j\geq i\geq 0$,
\item $a_{1,i}a_{\infty} = (a_{0,i}+\lambda a_{1,i})a_\lambda = 0$ for all $i\geq 0$ and for all $\lambda\in\K$,
\item $a_\infty x = a_\lambda x = b_{k,j}x = 0$ for all $j\geq 0$, for all $\lambda\in\K$, for $k\in\{1, 2\}$ and for all $x\in A$.
\een
It is not hard to verify that the set $U$ of all the products of the form:
\ben
\item $a_{l_1,i_1}\cdots a_{l_n,i_n}$ with $n\geq 0$ and $l_\nu\in\{0,1\}$, $i_\nu\geq 0$ for all $\nu\in\{1, \dots, n\}$,
\item $a_{l_1,i_1}\cdots a_{l_n,i_n}a_\mu$ with $n\geq 0$, $l_\nu\in\{0,1\}$, $i_\nu\geq 0$ for all $\nu\in\{1, \dots, n\}$, $l_n = 0$ and
$\mu\in\K\cup\{\infty\}$,
\item $a_{l_1,i_1}\cdots a_{l_n,i_n}a_0$ with $n>0$, $l_\nu\in\{0,1\}$, $i_\nu\geq 0$ for all $\nu\in\{1, \dots, n\}$ and $l_n = 1$,
\item $a_{l_1,i_1}\cdots a_{l_n,i_n}b_{k,j}$ with $n\geq 0$, $l_\nu\in\{0,1\}$, $i_\nu, j\geq 0$ for all $\nu\in\{1, \dots, n\}$, $k\in\{1, 2\}$ and
if $n > 0$ and $j\geq i_n$ then $l_n = 0$ and $k = 1$,
\een
is a $\K$-basis for $R$.

Let $\alpha\in R$, then $\alpha = \ds{\sum_{u\in U}} \alpha(u)u$, where $\alpha(u)\in \K$ and $\alpha(u) = 0$ for almost all $u\in U$. We define the
support of $\alpha$, $Supp(\alpha)$, to be $Supp(\alpha) = \{u\in U\mid \alpha(u)\neq 0\}$.

Cedó in \cite[Example 2]{Cedo} proved that $R$ is right zip. We shall see that $S = R[[x]]$ is not right zip. Let $X = \{a_{0,i}-a_{1,i}x\mid i\geq 0\}$
and $X_n = \{a_{0,i}-a_{1,i}x\mid n\geq i\geq 0\}$. It is easy to see that, for all $n\geq 0$, $b_{1,n}-b_{2,n}+b_{1,n}x+b_{2,n}x^2\in r.ann_S(X_n)$,
so for every finite subset $F$ of $X$ we have that $r.ann_S(F)\neq 0$. Let us see now that $r.ann_S(X) = \{0\}$.

Suppose that $r.ann_S(X)\neq \{0\}$. Then, there exists $\alpha = \ds{\sum_{i\geq 0}} \alpha_i x^i\in r.ann_S(X)$ such that $\alpha_0\neq 0$.
Since $a_{0,i}\alpha_0 = 0$ for all $i\geq 0$, we have that $\alpha_0 = \alpha_0(a_0)a_0$ (see the proof of \cite[Example 2]{Cedo} for details). Now,
$a_{0,i}\alpha_1 = \alpha_0(a_0)a_{1,i}a_0$, but $a_{1,i}a_0\not\in Supp(a_{0,i}\alpha_1)$, which is a contradiction. Therefore,
$r.ann_S(X) = \{0\}$, so $S$ is not right zip.
\end{proof}

Note that in Example \ref{ejemplocedo}, both $R[x]$ and $R[[x]]$ satisfy the right \Beachy condition. Moreover, in the rest of this section, we
will see that, under certain conditions, the \Beachy condition passes to power series extensions.

\begin{Def}
\tab Let $R$ be a ring and $\alpha$ be an endomorphism of $R$. We say that $R$ is strongly $\alpha$-skew Armendariz if for all
$\ds{f(x)=\sum_{i\geq 0} a_ix^i}$, $\ds{g(x)=\sum_{j\geq 0} b_jx^j}$ $\in R[[x;\alpha]]$ such that $f(x)g(x)=0$ we have that
$a_i\alpha^i(b_j)=0$ for all $i,j\geq 0$. When $\alpha$ denotes the identity of $R$, we say that $R$ is strongly Armendariz.
\end{Def}

It is clear that if a ring is strongly $\alpha$-skew Armendariz, then it is $\alpha$-skew Armendariz, but the converse is not true in general.

\begin{Ej}\label{ejemplo Armendariz}
\tab There exists an Armendariz ring which is not strongly Armendariz.
\end{Ej}

\begin{proof}
Let $\K = \Z_2$ and let $R$ be the $\K$-algebra presented with generators $\{a_i, b_j\mid i,j\geq 0\}$ and with relations
\renewcommand{\labelenumi}{(\alph{enumi})}
\ben
\item $a_ib_0 = a_{i-1}b_1+ a_{i-2}b_2 + \dots + a_0b_i$, for all $i\geq 1$,
\item $a_0b_0 = b_ia_j = a_ia_j = b_ib_j = 0$, for all $i,j\geq 0$.
\een

Let $B_a = \{a_i\mid i\geq 0\}$, $B_b = \{b_j\mid j\geq 0\}$ and $B_2 = \{a_ib_j\mid a_ib_j
\mid i\geq 0, j\geq 1\}$, then it is easy to check that $B = \{1\}\cup B_a\cup B_b\cup B_2$ is a $\K$-basis for $R$.
For all $r\in R$, $r = \ds{\sum_{z\in B}}r(z) z$, where $r(z)\in \K$ and $r(z) = 0$ for almost all $z\in B$. We define the support of $r$,
$Supp(r)$, to be $Supp(r) = \{z\in B\mid r(z)\neq 0\}$.

If we denote by $U_i$ the set of all finite sums of elements of $B_i$, with $i\in \{a,b,2\}$, then $R = U_0\oplus U_a\oplus U_b\oplus U_2$,
with $U_0 = \K$, and every element $r\in R$ can be written as $r = r_0+r_a+r_b+r_2$ with $r_i\in U_i$.

In order to continue the proof, we need the following technical lemma.
\begin{Lem}\label{lema tecnico}
\tab Let $f(x) = \ds{\sum_{i=0}^n} r_ix^i$, $g(x) = \ds{\sum_{j=0}^m} s_jx^j$ $\in R[x]\binv\{0\}$ be such that $f(x)g(x)=0$. Assume that
$r_i = r_{i,0}+r_{i,a}+r_{i,b}+r_{i,2}$ and $s_j = s_{j,0}+s_{j,a}+s_{j,b}+s_{j,2}$ with $r_{i,k},s_{j,k}\in U_k$ for all $i,j\geq 0$ and for all
$k\in\{0,a,b,2\}$. Then, $r_{i,0} = s_{j,0} = 0$ for all $i,j\geq 0$.
\end{Lem}

\par\noindent{\bf Proof of the Lemma. }
We have that $f(x)g(x) = \ds{\sum_{i=0}^n\sum_{j=0}^m}\ r_is_j\ x^{i+j} = 0$. Then:
\renewcommand{\labelenumi}{(\arabic{enumi})}
\ben
\item $\ds{\sum_{i=0}^n\sum_{j=0}^m}\ r_{i,0}s_{j,0}\ x^{i+j} = 0$
\item $\ds{\sum_{i=0}^n\sum_{j=0}^m}\ (r_{i,0}s_{j,a}+r_{i,a}s_{j,0})\ x^{i+j} = 0$
\item $\ds{\sum_{i=0}^n\sum_{j=0}^m}\ (r_{i,0}s_{j,b}+r_{i,b}s_{j,0})\ x^{i+j} = 0$
\item $\ds{\sum_{i=0}^n\sum_{j=0}^m}\ (r_{i,0}s_{j,2}+r_{i,a}s_{j,b}+r_{i,2}s_{j,0})\ x^{i+j} = 0$
\een

Assume that there exists some $i$ such that $r_{i,0}\neq 0$ and $i_1$ is the minimum with this property. If there exists $j$ such that $s_{j,0}\neq 0$,
assuming that $j_1$ is minimum with this property, then, by (1), we have that $r_{i_1+j_1,0}s_{0,0} + \dots + r_{i_1,0}s_{j_1,0} + \dots + r_{0,0}
s_{i_1+j_1,0} = r_{i_1,0}s_{j_1,0} = 0$, but this is a contradiction by the definition of $i_1$ and $j_1$. Therefore, $s_{j,0} = 0$ for all $j\geq 0$.

Now we have that:
\ben
\item[(2)] $\ds{\sum_{i=0}^n\sum_{j=0}^m}\ r_{i,0}s_{j,a}\ x^{i+j} = 0$
\item[(3)] $\ds{\sum_{i=0}^n\sum_{j=0}^m}\ r_{i,0}s_{j,b}\ x^{i+j} = 0$
\item[(4)] $\ds{\sum_{i=0}^n\sum_{j=0}^m}\ (r_{i,0}s_{j,2}+r_{i,a}s_{j,b})\ x^{i+j} = 0$
\een

Assume that there exists $j_2$ such that $s_{j_2,a}\neq 0$ and $j_2$ is the minimum with this property. Then, by (2),
$r_{0,0} s_{i_1+j_2,a} +\dots+ r_{i_1,0}s_{j_2,a} +\dots+ r_{i_1+j_2,0} s_{0,a} = r_{i_1,0}s_{j_2,a} = s_{j_2,a} = 0$, but this is a contradiction by
the definition of $i_1$ and $j_2$. Therefore, $s_{j,a} = 0$ for all $j\geq 0$. Analogously, by (3), we have that $s_{j,b} = 0$ for all $j\geq 0$.

Now we have that:
\ben
\item[(4)] $\ds{\sum_{i=0}^n\sum_{j=0}^m}\ r_{i,0}s_{j,2}\ x^{i+j} = 0$
\een
and, similarly as above, if there exists $j_3\geq 0$ such that $s_{j_3,2}\neq 0$ and $j_3$ is the minimum with this property. Then, by (4),
$r_{0,0} s_{i_1+j_3,2} +\dots+ r_{i_1,0}s_{j_3,2} +\dots+ r_{i_1+j_3,0} s_{0,2} = r_{i_1,0}s_{j_3,2} = s_{j_3,2} = 0$, but this is a contradiction by
the definition of $i_1$ and $j_3$. Therefore, $s_{j,2} = 0$ for all $j\geq 0$, and so, $g(x) = 0$, but this is a contradiction.
Then, $r_{i,0} = 0$ for all $i\geq 0$. Similarly we can see that $s_{j,0} = 0$ for all $j\geq 0$.
This completes the proof of the Lemma. $\cqd$\par\bigskip

% Fin demostración lema.

Now we continue the proof of Example~\ref{ejemplo Armendariz}.

Define $\ds{f'(x) = \sum_{i\geq 0} a_ix^i}$ and $\ds{g'(x)=\sum_{j\geq 0}b_jx^j}\in R[[x]]$. By (a), it is clear that $f'(x)g'(x)=0$.
However, $a_1b_0\neq 0$, and then, $R$ is not strongly Armendariz.

We shall see that $R$ is Armendariz. Let $f(x),g(x)\in R[x]\binv\{0\}$ be such that $f(x)g(x) = 0$. Assume that $f(x)=\ds{\sum_{i=0}^n} r_ix^i$ and
$g(x)=\ds{\sum_{j=0}^m} s_jx^j$. We want to see that $r_is_j = 0$ for all $0\leq i\leq n$ and for all $0\leq j\leq m$.
Assume that $r_i = r_{i,0} + r_{i,a} + r_{i,b} + r_{i,2}$ and $s_j = s_{j,0} + s_{j,a} + s_{j,b} + s_{j,2}$, with $r_{i,k}, s_{j,k}\in U_k$ for all
$i,j\geq 0$ and for all $k\in \{0,a,b,2\}$.

By Lemma~\ref{lema tecnico}, we may assume that $r_{i,0} = s_{j,0} = 0$ for all $0\leq i\leq n$ and for all $0\leq j\leq m$. Then,
$r_is_j = r_{i,a}s_{j,b}$ for all $0\leq i\leq n$ and for all $0\leq j\leq m$, so we may assume that $r_i\in U_a$ and $s_j\in U_b$ for all
$0\leq i\leq n$ and for all $0\leq j\leq m$. Without loss of generality we may also assume that $r_0,r_n,s_0,s_m\neq 0$.

We define the length of an element $r\in R\setminus\{0\}$ by $$l(r) =\max\{l(u)\mid u\in Supp(r)\}$$
where
$\left\{\begin{array}{lll}
   l(a_i) & = &  i\\
   l(b_j) & = &  j\\
   l(a_ib_j) & = & i+j \\
   l(1) & = & -1
  \end{array}\right.$,

and $l(0) = -\infty$.\\

We define the map $\delta$ over elements $r_a\in U_a\setminus\{0\}$, $r_b\in U_b\setminus\{0\}$ and $r_2\in U_2\setminus\{0\}$, by
$\delta(r_a) = a_{l(r_a)}$, $\delta(r_b) = b_{l(r_b)}$ and $\delta(r_2) = a_ib_j$, where $a_ib_j\in Supp(r_2)$ and $j$ is the biggest satisfying
$i+j = l(r_2)$.

We claim that, if $r\in U_a\binv\{0\}$ and $s\in U_b\binv\{0\}$ are such that $rs = 0$, then $r = a_0$ and $s = b_0$. Let us prove this claim.
Assume $r = \ds{\sum_{j=1}^n} a_{i_j}$ and $s = \ds{\sum_{l=1}^m} b_{k_l}$ with $0\leq i_1 < i_2 < \dots < i_n$ and $0\leq k_1 < k_2 < \dots < k_m$.
Assume $a_{i_n}b_{k_m}\neq 0$. Since $rs = 0$, we have to cancel $a_{i_n}b_{k_m}$ with another monomial of the form $a_{i_j}b_{k_l}$, but, since the
relations in (a) preserve the length of the elements of $R$, $l(a_{i_n}b_{k_m}) = i_n+k_m > l(rs-a_{i_n}b_{k_m})$, which is a contradiction.
Therefore, $a_{i_n}b_{k_m} = 0$, so $i_n = k_m = 0$, and then, $r = a_0$ and $s = b_0$.

Now, since $f(x)g(x) = 0$, by defining $r_i = s_j = 0$ for all $i>n$ and $j>m$, we have that $c_k = \ds{\sum_{i=0}^k r_is_{k-i} = 0}$, for all
$k\in\{0,\dots, n+m\}$. In particular $r_0s_0 = 0$, so, by the claim, $r_0 = a_0$ and $s_0 = b_0$.

Let $I_0 = \{0\leq i\leq n\mid r_i\neq 0\}\supseteq\{0,n\}$ and $J_0 = \{0\leq j\leq m\mid s_j\neq 0\}\supseteq\{0,m\}$.
We know that $l(r_i),l(s_j)\geq 0$ for all $i\in I_0$ and for all $j\in J_0$.
Assume $l(r_i) = 0$ for all $i\in I_0$, then $r_i = a_0$ for all $i\in I_0$ and
$$c_k = \ds{\sum_{i=0}^k a_0s_{k-i} = a_0 (\sum_{j=0}^k s_j) = 0}$$
for all $k\geq 0$. By the claim, $\ds{\sum_{j=0}^k} s_j = b_0$ for all $k\geq 0$, so $f(x) = \ds{\sum_{i\in I_0}} a_0 x^i$ and $g(x) = b_0$. Therefore,
$r_is_j = 0$ for all $0\leq i\leq n$ and for all $0\leq j\leq m$. Analogously, if $l(s_j) = 0$ for all $j\in J_0$, we have that $r_is_j = 0$ for all
$0\leq i\leq n$ and for all $0\leq j\leq m$.

Assume now that there exist $i\in I_0\setminus\{0\}$ and $j\in J_0\setminus\{0\}$ such that $l(r_i), l(s_j) > 0$. We shall see that this is impossible.
Let $i_1 > 0$ and $j_1 > 0$ be such that $l(r_i)\leq l(r_{i_1})$ for all $i\in I_0$ and $l(s_j)\leq l(s_{j_1})$ for all $j\in J_0$, and $i_1,j_1$ are
the minimum with this property. Now, since $c_{i_1+j_1} = r_0s_{i_1+j_1}+ \dots + r_{i_1}s_{j_1} + \dots + r_{i_1+j_1}s_0 = 0$ and
$r_{i_1}s_{j_1}\neq 0$, we need to cancel the monomial $\delta(r_{i_1}s_{j_1})$ in this expression.
By the definition of $i_1$ and $j_1$, we have that $l(r_{i_1}s_{j_1}) = l(\delta(r_{i_1}s_{j_1}))\geq l(r_ks_{i_1+j_1-k})$ for all
$1\leq k\leq i_1+j_1$. Assume that there exists $k\neq i_1$ such that $l(r_{i_1}s_{j_1}) = l(r_{i_1})+l(s_{j_1}) = l(r_ks_{i_1+j_1-k}) =
l(r_k) + l(s_{i_1+j_1-k})$. Then, $l(r_{i_1}) = l(r_k)$ and $l(s_{j_1}) = l(s_{i_1+j_1-k})$. By minimality of $i_1$, we
have that $k>i_1$, so $i_1+j_1-k < j_1$, which is a contradiction with the minimality of $j_1$.

Therefore, $R$ is an Armendariz ring.
\end{proof}

Let $\alpha$ be an automorphism of $R$. Let $\Delta^*$ be the set of all left annihilators of $R[[x;\alpha]]$, $\Delta^* = \{l.ann_{R[[x;\alpha]]}(V)
\mid V\subseteq R[[x;\alpha]]\}$ and recall that $\Gamma = \{l.ann_R(U)\mid U\subseteq R\}$.

If $U\subseteq R$, then $R[[x;\alpha]]l.ann_R(U) = l.ann_{R[[x;\alpha]]}(U)$, and, if $V\subseteq R[[x;\alpha]]$, we have that
$l.ann_{R[[x;\alpha]]}(V)\cap R = l.ann_R(C_V)$, where $C_V = \ds{\bigcup_{f(x)\in V}C_f}$ and, if $f(x)=\ds{\sum_{i\geq 0} a_ix^i}$,
$C_f = \{a_0,a_1,\dots,a_n,\dots\}\cup\{0\}$.
Therefore, we can define the maps $\phi^*: \Gamma\to\Delta^*$ by $\phi^*(A) = R[[x;\alpha]]A$ for all $A\in\Gamma$, and $\psi^*:\Delta^*\to\Gamma$ by
$\psi^*(B)=B\cap R$ for all $B\in\Delta^*$. It is easy to see that $\phi^*$ is injective and that $\psi^*$ is surjective. Moreover, $\phi^*$ is
bijective if and only if $\psi^*$ is bijective, and, in this case, one is the inverse of the other. Cortes in \cite[Lemma 2.7]{Cortes2} proved that
$\phi^*$ is bijective if and only if $R$ is strongly $\alpha$-skew Armendariz.

Cortes also proved that, although the zip property is not hereditary, when the skew power series ring of a ring $R$ is left zip, the ring $R$ itself is
left zip (\cite[Theorem 2.8]{Cortes2}). Someone could think that, analogously as we did in Lemma \ref{lema1}, a similar result for rings with the
\Beachy condition can be proven by requiring the ring to be $\alpha$-compatible. However, this is not true, as we will see in Section \ref{ejemplos}
(see Proposition \ref{prob2}).

Again Cortes in \cite[Theorem 2.8]{Cortes2} proved that, when the ring $R$ is strongly $\alpha$-skew Armendariz, $R$ is left zip if and only if
$R[[x;\alpha]]$ is left zip. A similar result is also true for rings with the \Beachy condition.

\begin{Teo}\label{Beachy series}
\tab Let $R$ be a strongly $\alpha$-skew Armendariz ring, with $\alpha$ an automorphism of $R$. If $R$ satisfies the left \Beachy condition,
then $R[[x;\alpha]]$ satisfies the left \Beachy condition.
\end{Teo}

\begin{proof}
Denote by $S$ the skew power series ring, $S = R[[x;\alpha]]$. Suppose that $R$ satisfies the \Beachy condition. Let $J\subseteq S$ be a left ideal of
$S$ such that $l.ann_S(J)=\{0\}$. Then, if $J' = JS$ is the (two-sided) ideal generated by $J$, we have that $l.ann_S(J') = \{0\}$. Then,
$l.ann_R(C_{J'})= \psi^*(l.ann_S(J')) = \psi^*(\{0\}) = \{0\}$ and $C_{J'}\subseteq R$. We shall see that $C_{J'}$ is an ideal of $R$.
\renewcommand{\labelenumi}{(\alph{enumi})}
\ben
\item If $s\in C_{J'}$, then there exists $f(x) = \ds{\sum_{i\geq 0}} a_ix^i\in J'$ such that $s = a_i$ for some $i$.
Let $r\in R$. Since $J'$ is an ideal of $S$, $g(x) = rf(x)$, $h(x) = f(x)\alpha^{-i}(r)\in J'$, and the coefficients of
$x^i$ in $g(x)$ and $h(x)$ are $ra_i$ and $a_ir$ respectively, so $rs, sr\in C_{J'}$.
\item If $r_1,r_2\in C_{J'}$, there exist $f(x),g(x)\in J'$ such that $\ds{f(x) = \sum_{i\geq 0} a_ix^i}$, $\ds{g(x) = \sum_{j\geq 0} b_jx^j}$ and
$r_1 = a_i$, $r_2 = b_j$ for some $i,j\geq 0$. We want to see that $r_1+r_2\in C_{J'}$. Assume without loss of generality that $i\leq j$.
Since $f(x),g(x)\in J'$ and $J'$ is an ideal of $S$, we have that $h(x)=f(x)x^{j-i}+g(x)\in J'$, and the coefficient in $x^j$ of $h(x)$ is $a_i + b_j$,
so $r_1+r_2 \in C_{J'}$.
\een

Now, since $R$ satisfies the left \Beachy condition, there exists a finite subset $X\subseteq C_{J'}$ such that $l.ann_R(X)=\{0\}$.
Assume that $X = \{a_1,\dots,a_n\}$, then, for every $a_i$ there exists a power series $f_i(x)\in J'$ such that $a_i\in C_{f_i}$.
Let $Y = \{f_1(x),\dots,f_n(x)\}\subseteq J'$. Clearly $X\subseteq C_Y$, so $l.ann_R(C_Y)\subseteq l.ann_R(X) = \{0\}$.
Since $R$ is strongly $\alpha$-skew Armendariz and
$$l.ann_R(C_Y) = \psi^*(l.ann_S(Y)) = \{0\},$$
by \cite[Lemma 2.7]{Cortes2}, we have that $l.ann_S(Y) = \{0\}$.

By the definition of $J'$, there exist integers $m_1,\dots, m_n\geq 0$ and power series $f_{i,j}(x)\in J$, $s_{i,j}(x)\in S$ for all $1\leq i\leq n$,
$0\leq j\leq m_i$ such that
$$f_i(x) = \sum_{j=0}^{m_i} f_{i,j}(x) s_{i,j}(x).$$
\tab Let $Y' = \{f_{i,j}(x)\mid 1\leq i\leq n\text{ and } 0\leq j\leq m_i\}\subseteq J$. Clearly, $l.ann_S(Y')\subseteq l.ann_S(Y) = \{0\}$.
Therefore, $S$ satisfies the left \Beachy condition.
\end{proof}

An easy consequence of Theorem \ref{Beachy series} is the following:
\begin{Cor}\label{Corbeachy2}
\tab Let $R$ be a strongly Armendariz ring. If $R$ satisfies the left \Beachy condition then $R[[x]]$ satisfies the left \Beachy condition.
\end{Cor}

Although, by Example \ref{ejemplocedo}, there exists a right zip ring $R$ such that $R[[x]]$ is not right zip, it is still an open problem whether
there exists such an example for the \Beachy condition or not. However, it would be really surprising that this example did not exist.

\section{Examples}\label{ejemplos}
\renewcommand{\labelenumi}{(\alph{enumi})}
\renewcommand{\theenumi}{(\alph{enumi})}

\tab In this section we give a negative answer to some interesting questions about the behavior of the zip property and the \Beachy condition.
These questions are the following:
\ben
\renewcommand{\labelenumi}{(Q\arabic{enumi})}
\renewcommand{\theenumi}{(Q\arabic{enumi})}
\item\label{q1} Let $R$ be a ring. Does $R$ being a left zip ring implies $R$ to be a right zip ring?
\item\label{q2} Let $R$ be a ring and $\alpha$ an automorphism of $R$. Does $R[[x;\alpha]]$ satisfying the left \Beachy condition implies $R$ to satisfy
the left \Beachy condition? What if $R$ is $\alpha$-compatible?
\een

Handelman and Lawrence gave an example of a prime ring in which every (faithful) left ideal is cofaithful but which does not have the analogous
property for right ideals (see \cite[Example 1]{HanLaw}). Therefore, there exist examples of rings satisfying the left \Beachy condition that do not
satisfy the right \Beachy condition. However, for the zip property, there are not explicit examples in the existing literature answering \ref{q1}.
We will see in Proposition \ref{prob1} that the answer to \ref{q1} is negative.

In Lemma~\ref{lema1}, we proved that if the polynomial extension of a ring that is $\alpha$-compatible, satisfies the left \Beachy condition, then
the ring itself satisfies the left \Beachy condition, and Cortes in \cite[Theorem 2.8]{Cortes2} proved that if $R[[x,\alpha]]$ is left zip, then $R$ is
left zip as well. However, the answer to \ref{q2} is negative, as we will see in Proposition \ref{prob2}.

Let $\K$ be a field, $A = \{a_i\mid i\geq 0\}$ and $B = \{b_j\mid j\geq 0\}$. Let $R$ be the $\K$-algebra presented with set of generators $A\cup B$
and with relations:
\ben
\renewcommand{\labelenumi}{(c\arabic{enumi})}
\renewcommand{\theenumi}{(c\arabic{enumi})}
\item\label{rel 1} $b_jb_l = a_ib_j = 0$ for all $i,j,l\geq 0$,
\item\label{rel 2} $b_ja_i = 0$ if and only if $j\geq i$.
\een

We denote by $\langle U\rangle$, with $U\subseteq R$, the multiplicative subsemigroup generated by the elements in $U$. Let $R_A = \K +
\K[\langle A\rangle]$. Note that $R_A\subseteq R$ is an integral domain.

Let $V = \{b_ja_{i_1}\cdots a_{i_n}\mid n\geq 1, i_1,\dots, i_n\geq 0\text{ and } j<i_1\}$. Let $U_a$, $U_b$ and $U_{ba}$ be the $\K$-linear span of
$\langle A\rangle$, $B$ and $V$ respectively. Clearly, $\mathfrak{B} = \{1\}\cup \langle A\rangle\cup B\cup V$ is a $\K$-basis of $R$,
and, for every $r\in R$, there exist unique $r_0\in\K$, $r_a\in U_a$, $r_b\in U_b$ and $r_{ba}\in U_{ba}$ such that $r = r_0+r_a+r_b+r_{ba}$.
For all $r\in R$, $r = \ds{\sum_{z\in \mathfrak{B}}} r(z) z$, where $r(z)\in \K$ and $r(z) = 0$ for almost all $z\in\mathfrak{B}$. We
define the support of $r$, $Supp(r)$, to be $Supp(r) = \{z\in\mathfrak{B}\mid r(z)\neq 0\}$.

Let $S = R[[x]]$. We denote by $U[[x]]$, with $U\subseteq R$, the set of all power series in $S$ with coefficients in $U$. For all $f(x)\in S$, there
exist unique $f_0(x)\in \K[[x]]$, $f_a(x)\in U_a[[x]]$, $f_b(x)\in U_b[[x]]$ and $f_{ba}(x)\in U_{ba}[[x]]$ such that
$f(x) = f_0(x) + f_a(x) + f_b(x) + f_{ba}(x)$. Note that $R_A[[x]]\subseteq R[[x]]$ is an integral domain and
$S = R_A[[x]]\oplus BS$, so, for every $f(x)\in S$, there exist unique $f_{A}(x)\in R_A[[x]]$ and
$f_B(x)\in BS$ such that $f(x) = f_A(x) + f_B(x)$. It is clear that $f_A(x) = f_0(x) + f_a(x)$ and $f_B(x) =  f_b(x) + f_{ba}(x)$.

\begin{Lem}\label{lema apart}
\ben
\item\label{apart 2} Let $r\in U_{ba}$ and $s\in \K\oplus U_a$. If $r,s$ are non-zero then $rs\neq 0$.
\item\label{apart 3} $l.ann_R(A) = \{0\}$.
\item\label{apart 4} Let $r = r_0+r_a+r_b+r_{ba}\in R$, with $r_0\in\K$, $r_a\in U_a$, $r_b\in U_b$ and $r_{ba}\in U_{ba}$. If $r_0$ is non-zero then
$l.ann_R(r) = r.ann_R(r) = \{0\}$.
\een
\end{Lem}

\begin{proof}
\ben
\item Let $r\in U_{ba}\binv\{0\}$ and $s\in \K\oplus U_a\binv\{0\}$. Then, there exist $w_1 = b_ja_{i_1}\cdots a_{i_n}\in Supp(r)$ and
$w_2 = a_{k_1}\cdots a_{k_m}\in Supp(s)$ such that the total degree in all the generators in $A$ is maximum, or $w_2 = 1$ if $s\in \K$. It is clear
that if $s\in \K\binv\0$, then $rs\neq 0$, so we may assume that $s\not\in \K$.

Assume $rs = 0$. Since $b_ja_{i_1}\cdots a_{i_n}a_{k_1}\cdots a_{k_m}\neq 0$, there exist $w'_1 = b_ja_{i'_1}\cdots a_{i'_{n'}}\in Supp(r)$ and
$w'_2\in Supp(s)$ such that
$$b_ja_{i_1}\cdots a_{i_n}a_{k_1}\cdots a_{k_m} = b_ja_{i'_1}\cdots a_{i'_{n'}}w'_2$$
and $(w_1,w_2)\neq(w'_1,w'_2)$.

By the choice of $w_1$ and $w_2$ we have that $n\geq n'$ and $m\geq deg_A(w'_2)$, where $deg_A(-)$ denotes the total degree in all the generators in
$A$. Thus, $n = n'$ and $deg_A(w'_2) = m$, so $i_p = i'_p$ and $w'_2 = w_2$ for all $1\leq p\leq n$, and then $(w_1,w_2) = (w'_1, w'_2)$, which is a
contradiction. Therefore, $rs\neq 0$.

\item Let $r\in l.ann_R(A)$. Assume $r = r_0+r_a+r_b+r_{ba}$ with $r_0\in\K$, $r_a\in U_a$, $r_b\in U_b$ and $r_{ba}\in U_{ba}$.
We have that $ra_i = r_0a_i + r_aa_i + r_ba_i + r_{ba}a_i = 0$ for all $i\geq 0$, so $(r_0 + r_a)a_i = 0$ and $(r_b + r_{ba})a_i = 0$ for all $i\geq 0$.
Since $R_A$ is an integral domain, we have that $r_0 + r_a = 0$, so $r_0 = r_a = 0$.

Now we have that $r_ba_i = -r_{ba}a_i$ for all $i\geq 0$. Suppose that $r_ba_i\neq 0$ for some $i\geq 0$. Then, there exist $b_la_i\in Supp(r_ba_i)$
and $wa_i \in Supp(r_{ba}a_i)$, with $w\in Supp(r_{ba})$, such that $b_la_i = wa_i\neq 0$, but this is impossible since $deg_A(b_la_i) = 1$ and
$deg_A(wa_i)\geq 2$, and the defining relations \ref{rel 1} and \ref{rel 2} preserve degrees in all the generators in $A$. Therefore,
$r_ba_i = r_{ba}a_i = 0$ for all $i\geq 0$. By \ref{apart 2}, since $r_{ba}a_0 = 0$ and $a_0\neq 0$, we have that $r_{ba} = 0$. Finally, we shall
see that $r_b = 0$.

Suppose $r_b\neq 0$ and let $r_b = \ds{\sum_{i=1}^n} \lambda_ib_{j_i}$ with $\lambda_i\in\K\binv\{0\}$ and $j_1<j_2<\dots <j_n$. Then, by \ref{rel 2},
$r_ba_{j_1+1} = \lambda_1 b_{j_1}a_{j_1+1}\neq 0$, which is a contradiction. Therefore, $r_b = 0$ and then $l.ann_R(A) = \{0\}$.

\item Let $r = r_0+r_a+r_b+r_{ba}\in R$ with $r_0\in\K$, $r_a\in U_a$, $r_b\in U_b$ and $r_{ba}\in U_{ba}$ be such that $r_0\neq 0$.
Let $s = s_0+s_a+s_b+s_{ba}\in l.ann_R(r)$ with $s_0\in\K$, $s_a\in U_a$, $s_b\in U_b$ and $s_{ba}\in U_{ba}$. Then, since $sr = 0$, we have that
$s_0r_0 = 0$, so $s_0 = 0$, and, by \ref{rel 1} and \ref{rel 2}, $sr = s_ar_0 + s_ar_a + s_br_0 + s_br_a + s_{ba}r_0 + s_{ba}r_a$. Now we have that
$s_ar_a + s_ar_0 = 0$, $s_br_0 = 0$, so $s_b = 0$, and $s_{ba}r_0 + s_{ba}r_a = 0$.

We have that $s_a(r_0+r_a) = 0$ so, since $R_A$ is an integral domain and $(r_0+r_a)\neq 0$, then $s_a = 0$. We also have that
$s_{ba}r_0 = -s_{ba}r_a$. If $s_{ba}\neq 0$, then $r_a\neq 0$ and $deg_A(s_{ba}r_0) = deg_A(s_{ba}) = deg_A(s_{ba}r_a) = deg_A(s_{ba})+deg_A(r_a) >
deg_A(s_{ba})$, which is a contradiction. Therefore $s = s_{ba} = 0$, so
$l.ann_R(r) = \{0\}$.

Let $s = s_0+s_a+s_b+s_{ba}\in r.ann_R(r)$ with $s_0\in\K$, $s_a\in U_a$, $s_b\in U_b$ and $s_{ba}\in U_{ba}$. Then, since $rs = 0$ we have that
$r_0s_0 = 0$, so $s_0 = 0$, and, by \ref{rel 1}, $rs = r_0s_a + r_as_a + r_0s_b + r_0s_{ba} + r_bs_a + r_{ba}s_a$. Now we have that $(r_0 + r_a)s_a = 0$,
$r_0s_b = 0$, so $s_b = 0$, and $r_0s_{ba} + r_bs_a + r_{ba}s_a = 0$.

Since $(r_0+r_a)s_a = 0$, $R_A$ is an integral domain and $(r_0+r_a)\neq 0$, we have that $s_a = 0$. Finally, we have that
$r_0s_{ba} = 0$, so $s_{ba} = 0$ and then, $s = 0$. Therefore, $r.ann_R(r) = \{0\}$.
\een
\end{proof}

\begin{Lem}\label{apart 5}
\tab Let $r_1,\dots,r_n\in R\binv\{0\}$ and $s\in U_a\binv\{0\}$. Then, $\ds{\sum_{i=1}^n}r_ia_is = 0$ if and only if $r_ia_is = 0$ for all
$1\leq i\leq n$.
\end{Lem}

\begin{proof}
Let $r_1,\dots,r_n\in R\binv\{0\}$ and $s\in U_a\binv\{0\}$ be such that $\ds{\sum_{i=1}^n}r_ia_is = 0$.

Suppose that there exists $1\leq i\leq n$ such that $deg_A(r_i)>0$. For each $1\leq i\leq n$, let $w_i\in Supp(r_i)$ be such that $deg_A(w_i)$ is
maximum. Let $i_0$ be such that $deg_A(w_{i_0}) = \max\{deg_A(w_i)\mid 1\leq i\leq n\}$. Let $u\in Supp(s)$ be such that $deg_A(u)$ is maximum.
Then, $0\neq w_{i_0}a_{i_0}u\in Supp(r_{i_0}a_{i_0}s)$. Since $\ds{\sum_{i=1}^n}r_ia_is = 0$, there exist $w\in \ds{\bigcup_{i = 1}^n} Supp(r_i)$,
$u'\in Supp(s)$ and $1\leq i_1\leq n$ such that $(w,i_1,u')\neq (w_{i_0},i_0,u)$ and $wa_{i_1}u' = w_{i_0}a_{i_0}u\neq 0$. Then,
$deg_A(wa_{i_1}u') = deg_A(w)+deg_A(u')+1 = deg_A(w_{i_0}a_{i_0}u) = deg_A(w_{i_0})+deg_A(u)+1$, so, by the definition of $w_{i_0}$ and $u$, we have
that $deg_A(w) = deg_A(w_{i_0})$ and $deg_A(u) = deg_A(u')$. Thus, $w = w_{i_0}$, $u = u'$ and $i_1 = i_0$, which is a contradiction.
Therefore, for all $1\leq i\leq n$, $deg_A(r_i) = 0$, so $r_i\in U_b\oplus \K$, $r_i\neq 0$, for all $1\leq i\leq n$.
Similarly, it is easy to see that $1\not\in Supp(r_i)$ for all $1\leq i\leq n$.

Then we have $r_i\in U_b$ for all $1\leq i\leq n$. Let $u\in Supp(s)$ be such that $deg_A(u)$ is maximum. Assume that there exists $i$ such that
$r_ia_is\neq 0$. Then, there exists $b_j\in Supp(r_i)$ such that $b_ja_iu\in Supp(r_ia_is)$. Since $\ds{\sum_{i=1}^n}r_ia_is = 0$, there exist
$1\leq i_1\leq n$, $b_l\in Supp(r_{i_1})$ and $u'\in Supp(s)$ such that $(j,i,u)\neq (l,i_1,u')$ and $b_ja_iu = b_la_{i_1}u'\neq 0$, but this is
impossible. Therefore, $r_ia_is = 0$ for all $1\leq i\leq n$.
\end{proof}

\begin{Lem}\label{apart 6}
\tab Let $f(x), g(x)\in R[[x]]\binv\{0\}$. If $f(x)g(x) = 0$ then $f_0(x) = g_0(x) = 0$, where $f(x) = f_0(x) + f_a(x) + f_b(x) + f_{ba}(x)$,
$g(x) = g_0(x) + g_a(x) + g_b(x) + g_{ba}(x)$, with $f_0(x),g_0(x)\in\K[[x]]$, $f_a(x),g_a(x)\in U_a[[x]]$, $f_b(x),g_b(x)\in U_b[[x]]$ and
$f_{ba}(x),g_{ba}(x)\in U_{ba}[[x]]$.
\end{Lem}

\begin{proof}
Let $f(x) = f_0(x) + f_a(x) + f_b(x) + f_{ba}(x)$, $g(x) = g_0(x) + g_a(x) + g_b(x) + g_{ba}(x)\in R[[x]]\binv\{0\}$ be such that $f(x)g(x) = 0$.
Let $f_A(x) = f_0(x) + f_a(x)$, $g_A(x) = g_0(x) + g_a(x)$ and $f_B(x) = f_{b}(x) + f_{ba}(x)$, $g_B(x) = g_b(x) + g_{ba}(x)$.
Then, we have that $f_0(x)g_0(x) = 0$ and $f_A(x)g_A(x) = 0$. Since $\K[[x]]$ and $R_A[[x]]$ are integral domains, we have that either
$f_A(x) = 0$ or $g_A(x) = 0$ and either $f_0(x) = 0$ or $g_0(x) = 0$.

Assume $f_0(x)\neq 0$. Then, $g_A(x) = 0$ and $f(x)g(x) = f_0(x)g_B(x) = 0$. Let $f_0(x) = \ds{\sum_{i\geq i_0}}\varepsilon_i x^i$, with
$\varepsilon_i\in \K$ for all $i\geq i_0$ and $\varepsilon_{i_0}\neq 0$. Since $g(x) = g_B(x)\neq 0$, we have that $g_B(x) = \ds{\sum_{j\geq j_0}}
r_j x^j$, with $r_j\in BR$ for all $j\geq j_0$ and $r_{j_0}\neq 0$. The coefficient of $x^{i_0+j_0}$ in $f_0(x)g_B(x)$ is $\varepsilon_{i_0}r_{j_0}
\neq 0$, which is a contradiction. Therefore, $f_0(x) = 0$.

Assume now that $g_0(x)\neq 0$. Then, $f_A(x) = 0$ and $f(x)g(x) = f_B(x)g_A(x) = f_B(x)g_0(x) + f_B(x)g_a(x)$. Therefore, $f_b(x)g_0(x) = 0$ and
$$f_{ba}(x)g_0(x) = -f_B(x)g_a(x).$$
Let $g_0(x) = \ds{\sum_{i\geq i_0}}\varepsilon_i x^i$, with $\varepsilon_i\in \K$ for all $i\geq i_0$ and
$\varepsilon_{i_0}\neq 0$. If $f_b(x)\neq 0$, then $f_b(x) = \ds{\sum_{j\geq j_0}} r_j x^j$ for some $r_j\in U_b$ for all $j\geq j_0$ and
$r_{j_0}\neq 0$. Now, the coefficient of $x^{i_0+j_0}$ in $f_b(x)g_0(x)$ is $r_{j_0}\varepsilon_{i_0}\neq 0$, which is a contradiction.
Therefore, $f_b(x) = 0$ and $f_{ba}(x)g_A(x) = 0$. Since $f(x) = f_{ba}(x)\neq 0$, $f_{ba}(x) = \ds{\sum_{j\geq j_1}} s_j x^j$, with $s_j\in U_{ba}$
for all $j\geq j_1$ and $s_{j_1}\neq 0$.
Since $g_A(x)\neq 0$, $g_A(x) = \ds{\sum_{i\geq i_1}} t_ix^i$, with $t_i\in\K \oplus U_a$ and $t_{i_1}\neq 0$. Now the coefficient of $x^{i_1+j_1}$ in
$f_{ba}(x)g_A(x)$ is $s_{j_1}t_{i_1} = 0$, which is a contradiction to \ref{apart 2} in Lemma \ref{lema apart}. Therefore, $g_0(x) = f_0(x) = 0$.
\end{proof}

\begin{Pro}\label{prob2}
  \tab The ring $R$ does not satisfy the left \Beachy condition. However, $R[[x]]$ satisfies the left \Beachy condition.
\end{Pro}

\begin{proof}
Let $I$ be the left ideal of $R$ generated by $A$. By \ref{apart 3} in Lemma \ref{lema apart}, $l.ann_R(A) = \{0\}$, so $l.ann_R(I) = \{0\}$. Let
$F = \{r_1,\dots,r_n\}$ be a finite subset of $I$. Then, there exist $m_1,\dots, m_n\geq 1$, $w_{i,j}\in \langle A\rangle\cup V$ (recall that
$V = \{b_ja_{i_1}\cdots a_{i_n}\mid n\geq 1, i_1,\dots, i_n\geq 0\text{ and } j<i_1\}$) and $\lambda_{i,j}\in\K \binv\{0\}$ for all $1\leq j\leq m_i$
and for all $1\leq i\leq n$ such that $r_i = \ds{\sum_{j = 1}^{m_i}}\lambda_{i,j} w_{i,j}$ and $w_{i,j}\neq w_{i,l}$ for all $1\leq i\leq n$ and for
all $j\neq l$. Let $X = \{a_{k_{i,j}}\mid w_{i,j} = a_{k_{i,j}}w'_{i,j}\text{ for some } 1\leq i\leq n, 1\leq j\leq m_i\text{ and }w'_{i,j}\in\langle A
\rangle\cup\{1\}\}$. If $X = \emptyset$, then $b_0r_i = 0$ for all $1\leq i\leq n$, so $b_0\in l.ann_R(F)$. Assume $X\neq\emptyset$ and let
$k = \max\{k_{i,j}\mid a_{k_{i,j}} \in X\}$. Then, by the defining relations \ref{rel 1} and \ref{rel 2}, $b_kr_i = 0$ for all $1\leq i\leq n$, so
$b_k\in l.ann_R(F)$. Therefore, $R$ does not satisfy the left \Beachy condition.

We shall see now that $S = R[[x]]$ satisfies the left \Beachy condition. Let $J$ be a left ideal of $S$ such that $l.ann_S(J) = \{0\}$.
If there exists $f(x)\in J$ such that $f_0(x)\neq 0$, then, by Lemma \ref{apart 6}, we have that $l.ann_S(f(x)) = \{0\}$.

Suppose that for all $f(x)\in J$, $f_0(x) = 0$ and let $g(x) = \ds{\sum_{i\geq 0}a_ix^i}$. Since $l.ann_S(J) = \{0\}$, there exists $f(x)\in J$ such that
$g(x)f(x)\neq 0$, and $h(x) = g(x)f(x) = g(x)f_a(x)\in J$, since $J$ is a left ideal of $S$. We shall see that $l.ann_S(h(x)) = \{0\}$.

Suppose that there exists $t(x)\in S\setminus\{0\}$ such that $t(x)h(x) = 0$. Since $h(x)\in R_A[[x]]\setminus\{0\}$, it is clear by the
proof of Lemma \ref{apart 6}, that $t(x) = t_b(x)+t_{ba}(x)$ ($t_A(x) = 0$, $t(x) = t_B(x)$). Let $f_a(x) = \ds{\sum_{i\geq i_0}} r_ix^i$, where
$r_i\in U_a$ and $r_{i_0}\neq 0$. Define $c_k = \ds{\sum_{j=i_0}^k} a_{k-j}r_j\in U_a$ for all $k\geq i_0$, then, $h(x) = \ds{\sum_{k\geq i_0} c_kx^k}$.
Suppose that $t(x) = \ds{\sum_{j\geq 0}}s_jx^j$. We may assume without loss of generality that $s_0\neq 0$ and $s_j = s_{j,b}+s_{j,ba}$, where
$s_{j,b}\in U_b$ and $s_{j,ba}\in U_{ba}$, for all $j\geq 0$. Then, for all $k\geq i_0$, we have that
$$\sum_{j = i_0}^k s_{k-j}c_j = 0.$$

%%% Caso 1.
For $k = i_0$ we have that $s_0c_{i_0} = s_0a_0r_{i_0} = s_{0,ba}a_0r_{i_0} = 0$, because, by the relation \ref{rel 2}, $s_{0,b}a_0 = 0$. Then, by
\ref{apart 2} in Lemma~\ref{lema apart}, $s_{0,ba} = 0$, so $s_0 = s_{0,b}\neq 0$.

%%% Caso 2.
For $k = i_0+1$ we have that $s_0c_{i_0+1} + s_1c_{i_0} = s_{0,b}(a_0r_{i_0+1} + a_1r_{i_0}) + s_1a_0r_{i_0} = s_{0,b}a_1r_{i_0} +
s_{1,ba}a_0r_{i_0} = 0$, since by the relation \ref{rel 2}, $s_{0,b}a_0 = s_{1,b}a_0 = 0$. Then, by Lemma \ref{apart 5}, $s_{0,b}a_1r_{i_0} =
s_{1,ba}a_0r_{i_0} = 0$, so, by \ref{apart 2} in Lemma~\ref{lema apart}, $s_{0,b}a_1 = 0$ and $s_{1,ba} = 0$.

%%% Hipotesis de induccion
\underline{Induction Hypothesis:} Assume that $s_{i,ba} = 0$, for all $i<n$, and $s_{j,b}a_{i-j} = 0$, for all $i<n$ and for all $0\leq j\leq i$.

%%% Prueba por induccion
We shall see that $s_{n,ba} = 0$ and that $s_{0,b}a_n = s_{1,b}a_{n-1} = \dots = s_{n,b}a_0 = 0$.

For $k = n+i_0$ we have that $\ds{\sum_{j = i_0}^{n+i_0}} s_{n+i_0-j}c_j = \ds{\sum_{j = 0}^n} s_{n-j}c_{j+i_0} = 0$.
By the induction hypothesis, we have that, for all $j>0$,
$$s_{n-j}c_{j+i_0} = s_{n-j,b}c_{j+i_0} = \sum_{k=0}^j s_{n-j,b}a_{j-k}r_{k+i_0} = s_{n-j,b}a_jr_{i_0}.$$

Thus, $\ds{\sum_{j = 0}^n} s_{n-j}c_{j+i_0} = \ds{\sum_{j=0}^n} s_{n-j}a_jr_{i_0} = 0$, so, by Lemma \ref{apart 5}, we have that $s_{n-j}a_jr_{i_0} = 0$
for all $0\leq j\leq n$. In particular, by \ref{rel 2}, $s_na_0r_{i_0} = s_{n,ba}a_0r_{i_0} = 0$, so, by \ref{apart 2} in Lemma~\ref{lema apart},
$s_{n,ba} = 0$. Moreover, by \ref{apart 2} in Lemma~\ref{lema apart}, since $s_{n-j}a_jr_{i_0} = s_{n-j,b}a_jr_{i_0} = 0$ for all $j>0$, we have that
$s_{n-j,b}a_j = 0$ for all $j> 0$ and $s_{n,b}a_0 = 0$ by the relation \ref{rel 2}.

Therefore, $s_j = s_{j,b}$ for all $j\geq 0$ and $s_ja_{i-j} = 0$ for all $i\geq 0$ and for all $0\leq j\leq i$.
In particular, $s_0a_i = 0$ for all $i\geq 0$, so $s_0\in l.ann_R(A)$, but, by \ref{apart 3} in Lemma \ref{lema apart}, $l.ann_R(A) = \{0\}$, which is a
contradiction. Thus, $l.ann_S(h(x)) = \{0\}$.

Therefore, $R[[x]]$ satisfies the left \Beachy condition.
\end{proof}
Note that every ring is $1_R$-compatible, where $1_R$ denotes the identity automorphism of $R$, and $R[[x; 1_R]] = R[[x]]$. Therefore, Proposition
\ref{prob2} answers \ref{q2} in the negative even in the case of $\alpha$-compatible rings.

\begin{Pro}\label{prob1}
   \tab The ring $R$ is right but not left zip.
\end{Pro}

\begin{proof}
By Proposition~\ref{prob2}, we know that $R$ does not satisfy the left \Beachy condition. Therefore, $R$ is not left zip.

Let $X\subseteq R$ be such that $r.ann_R(X) = \{0\}$. Suppose that $r_0 = 0$ for all $r = r_0 + r_a + r_b + r_ba\in X$. Then,
$rb_0 = r_ab_0 + r_bb_0 + r_{ba}b_0 = 0$ by the relation \ref{rel 1}, for all $r\in X$, so $b_0\in r.ann_R(X)$, which is a contradiction.
Therefore, there exists $r = r_0 + r_a + r_b + r_ba\in X$ such that $r_0\neq 0$. Then, by \ref{apart 4} in Lemma \ref{lema apart}, $r.ann_R(r) = \{0\}$,
so $R$ is a right zip ring.
\end{proof}

Note that $R[[x]]$ satisfies the left \Beachy condition but it is not left zip, since, by \cite[Theorem 2.8]{Cortes2}, if $R[[x]]$ is left zip then
$R$ is left zip as well, but we have seen that $R$ does not satisfy the left \Beachy condition and, therefore, is not left zip.

\vspace{30pt} \noindent
\begin{center}
\begin{tabular}{l}
Elena Rodríguez-Jorge \\
Department of Mathematics\\
Universitat Autònoma de Barcelona\\
08193 Bellaterra (Barcelona), Spain
\end{tabular}
\end{center}
\end{document}